\documentclass[twoside, 12pt]{article}
\usepackage{amssymb, amsmath, mathrsfs, amsthm}
\usepackage{enumerate}
\usepackage{graphicx}
\usepackage{color}
\usepackage[top=1in, bottom=1in, left=1in, right=1in]{geometry}
\usepackage{float, caption, subcaption}
\usepackage[dvipsnames]{xcolor}
\usepackage{tikz}
\usepackage{optidef}
\newtheorem{theorem}{Theorem}[section]

\newtheorem{corollary}[theorem]{Corollary}
\newtheorem{observation}[theorem]{Observation}
\newtheorem{proposition}[theorem]{Proposition}
\newtheorem{conjecture}[theorem]{Conjecture}

\theoremstyle{definition}

\newcommand\abs[1]{\left|#1\right|}
\renewcommand\deg[1]{\operatorname{deg}\left(#1\right)}
\newcommand\dist[2]{\operatorname{dist}\left(#1,#2\right)}
\newcommand\floor[1]{\lfloor#1\rfloor}

\DeclareMathOperator{\Z}{Z}
\DeclareMathOperator{\F}{F}
\DeclareMathOperator{\pt}{pt}
\DeclareMathOperator{\PT}{PT}
\DeclareMathOperator{\ft}{ft}

\begin{document}
\title{\bf On the forts and related parameters of the hypercube graph}
\author{Boris Brimkov\footnote{Department of Mathematics and Statistics, Slippery Rock University {\tt (boris.brimkov@sru.edu)}}, Thomas R. Cameron\footnote{Department of Mathematics, Penn State Behrend {\tt (trc5475@psu.edu)}}, Owen Grubbs\footnote{Department of Mathematics, Penn State Behrend {\tt (olg5037@psu.edu)}}}
\maketitle

\begin{abstract}
In 2018, forts were defined as non-empty subsets of vertices in a graph where no  vertex outside the set has exactly one neighbor in the set. Forts have since been used to characterize zero forcing sets, model zero forcing as an integer program, and provide lower bounds on the zero forcing number. In this article, we give a complete characterization of minimum forts in the hypercube graph, showing that they are automorphic to one of two sets. In contrast, non-automorphic minimum zero forcing sets are identified with distinct propagation times. We also derive the fractional zero forcing number and bounds on the fort number of the hypercube. When the hypercube's dimension is a power of two, the fort number and fractional zero forcing number are equal to the domination number, total domination number, and open packing number. Lastly, we present general constructions for minimal forts in the Cartesian product of graphs, reflecting some minimal forts of the hypercube.
\end{abstract}
\thispagestyle{empty}
\section{Introduction}\label{sec:intro}
Zero forcing is a binary coloring game played on a graph, where a set of filled vertices can force non-filled vertices to become filled according to a color change rule. A zero forcing set is a set of filled vertices that can force the entire graph; the zero forcing number is the minimum cardinality of a zero forcing set. In 2007, the concept of zero forcing was independently developed in the context of quantum control, where it was known as graph infection~\cite{Burgarth2007}. By 2008, researchers demonstrated that the zero forcing number serves as an upper bound on the maximum nullity of a real symmetric matrix associated with a graph~\cite{AIM2008}. In addition, the combinatorial optimization problem for the zero forcing number is NP-hard, as the corresponding decision problem was shown to be NP-complete~\cite{Aazami2008}. Since then, the study of zero forcing and its applications has gained significant attention~\cite{Hogben2022}.

Related graph parameters, such as propagation time and throttling number, have also been extensively explored in the literature. Various works have investigated known values for specific families of graphs, extreme values, the impacts of certain graph operations, and the NP-hardness of the combinatorial optimization problem corresponding to the throttling number~\cite{Brimkov2019th, Butler2013, Carlson2021, Chilakamarri2012, Hogben2012, Hogben2022}. Moreover, the applications of propagation time and throttling number extend to the controllability of quantum systems and the placement problem of phase measurement units in electrical engineering~\cite{Brueni2005, Severini2008}.

Failed zero forcing sets were introduced in 2014 as subsets of filled vertices that cannot force the entire graph~\cite{Fetcie2015}; the failed zero forcing number is the maximum cardinality of a failed zero forcing set. Since then, the failed zero forcing number has garnered considerable attention, with studies focusing on specific families of graphs, extreme values, the effects of various graph operations, and the NP-hardness of the combinatorial optimization problem corresponding to the failed zero forcing number~\cite{Afzali2024, Fetcie2015, Gomez2021, Gomez2025, Kaudan2023, Shitov2017, Swanson2023}.

In 2018, the forts of a graph were introduced; a fort is a non-empty subset of vertices such that no vertex outside the set has exactly one neighbor within it~\cite{Fast2018}. Forts are closely related to failed zero forcing sets, as a subset of vertices is a fort if and only if its complement is a stalled failed zero forcing set.
Therefore, the NP-hardness of finding a maximum failed zero forcing set implies the NP-hardness of finding a minimum fort of a graph. 
Moreover, every zero forcing set of a graph intersects every fort; that is, the zero forcing sets form a cover for the forts. Thus, forts can turn the iterative graph coloring process of zero forcing into a static set covering problem. This idea has been used in computational approaches for finding the minimum zero forcing set of a graph~\cite{Brimkov2019, Brimkov2021, Cameron2023, Fast2018}. The optimal value of the linear programming relaxation of this set covering model is referred to as the fractional zero forcing number, which is a lower bound on the zero forcing number. The maximum number of pairwise disjoint forts is called the fort number and is also a lower bound on the zero forcing number. 

In this article, we give a complete characterization of minimum forts of the hypercube graph; see Section~\ref{sec:minimum_forts}. From this characterization, we derive the fractional zero forcing number and establish both lower and upper bounds for the fort number of the hypercube graph. When the dimension of the hypercube is a power of two, the fort number and fractional zero forcing number are equivalent to the domination number, total domination number, and open packing number. We note that the fractional zero forcing number of the hypercube is already known, see~\cite[Proposition 3.15]{Cameron2023}. However, their proof relies on the failed zero forcing number of the hypercube from~\cite[Theorem 4.2]{Afzali2024}, which does not provide a characterization of all minimum forts. 

The hypercube graph is highly symmetric; for example, it is vertex and edge transitive, distance regular, and distance transitive~\cite{Mirafzal2025}. Thus, one might expect that its minimum forts are automorphic, which is almost always the case. We show that the minimum forts of the hypercube graph are open neighborhoods, except for dimension four, which has two classes of non-automorphic minimum forts. Even more surprisingly, we show that the minimum zero forcing sets of the hypercube graph are highly non-automorphic; see Section \ref{sec:minimum_zf_sets}. While the characterization of all possible minimum zero forcing sets of the hypercube remains an open problem, we demonstrate that the minimum zero forcing sets of the hypercube possess a rich structure that inherits non-automorphisms from hypercubes of lower dimension.
We note here that, despite its symmetry and simple definition, many classical problems posed on the hypercube have proven challenging to resolve. For example, the domination number, total domination number, crossing number, and number of independent sets in the hypercube graph remain unknown, despite extensive investigation~\cite{Azarija2017, Bertolo2004, Erdos1973, Jenssen2022}.

Lastly, in Section~\ref{sec:minimal_forts}, we present general constructions for the minimal forts of the Cartesian product of graphs. These lead to a partial characterization of the minimal forts of the hypercube.
\section{Preliminaries}\label{sec:prelim}
In this section, we introduce definitions and preliminary results that will be used throughout this article.

\subsection{Graph definitions}
Throughout this article, we let $\mathbb{G}$ denote the set of all finite simple unweighted graphs.
For each $G\in\mathbb{G}$, we have $G=(V,E)$, where $V$ is the vertex set, $E$ is the edge set, and $\{u,v\}\in E$ if and only if $u\neq v$ and there is an edge between vertices $u$ and $v$.
If the context is not clear, we use $V(G)$ and $E(G)$ to specify the vertex set and edge set of $G$, respectively. 
The \emph{order} of $G$ is denoted by $n=\abs{V}$ and the \emph{size} of $G$ is denoted by $m=\abs{E}$.
If $m=\binom{n}{2}$, then we refer to $G$ as the complete graph which we denote by $K_{n}$.
The graph $G$ is \emph{bipartite} if $V$ can be split into two disjoint and independent sets $V_{1}$ and $V_{2}$, that is, every edge of $G$ contains a vertex in $V_{1}$ and a vertex in $V_{2}$.
We refer to $V_{1}$ and $V_{2}$ as \emph{parts} of the bipartite graph $G$. 

Let $G\in\mathbb{G}$.
Given $u\in V$, we define the \emph{neighborhood} of $u$ as 
\[
N(u) = \left\{v\in V\colon\{u,v\}\in E\right\}.
\]
We refer to every $v\in N(u)$ as a \emph{neighbor} of $u$ and we say that $u$ and $v$ are \emph{adjacent}. 
If the context is not clear, we use $N_{G}(u)$ to denote the neighborhood of $u$ in the graph $G$.
The \emph{closed neighborhood} of $u$ is defined by $N[u] = N(u)\cup\{u\}$. 
The degree of $u$ is denoted by $\deg{u}=\abs{N(u)}$.
A \emph{graph automorphism} is a permutation of the vertex set that preserves adjacency.
Two subsets $S,S'\subset V$ are considered \emph{automorphic} if there exists a graph automorphism that maps one set to the other.  

Given vertices $u,v\in V$, we say that $u$ and $v$ are \emph{connected} if there exists a list of distinct vertices $(w_{0},w_{1},\ldots,w_{l})$ such that $u=w_{0}$, $v=w_{l}$, and $\{w_{i},w_{i+1}\}\in E$ for all $i=1,\ldots,l-1$.
Such a list of vertices is called a \emph{path}; in particular, we reference it as a $(u,v)$-path. 
Furthermore, the \emph{length} of the $(u,v)$-path is given by $l$. 
If $u$ is connected to $v$, then the \emph{distance} from $u$ to $v$, denoted $\dist{u}{v}$, is the length of the shortest $(u,v)$-path. For other graph terminology and notation, we refer the reader to \cite{West2001}.

\subsection{Zero forcing}
\emph{Zero forcing} is a binary coloring game on a graph, where vertices are either filled or non-filled. 
In this paper, we denote filled vertices by the color gray and non-filled vertices by the color white. 
An initial set of gray vertices can force white vertices to become gray following a color change rule. 
While there are many color change rules, see~\cite[Chapter 9]{Hogben2022}, we will use the \emph{standard rule} which states that a gray vertex $u$ can force a white vertex $v$ if $v$ is the only white neighbor of $u$.

A \emph{time step} refers to the application of all possible forces that can be done independently of each other. 
Since the vertex set is finite, there comes a point in which no more forcings can be applied. 
If at this point all vertices are gray, then we say that the initial set of gray vertices is a \emph{zero forcing set} of $G$; otherwise, we refer to the initial set of gray vertices as a \emph{failed zero forcing set} of $G$.
The \emph{zero forcing number} of $G$, denoted $\Z(G)$, is the minimum cardinality of a zero forcing set of $G$.
The \emph{failed zero forcing number} of $G$, denoted $\F(G)$, is the maximum cardinality of a failed zero forcing set of $G$.
The \emph{minimum propagation time} of $G$, denoted $\pt(G)$, is the smallest number of time steps necessary for a minimum zero forcing set to force the entire vertex set gray. 
The \emph{maximum propagation time} of $G$, denoted $\PT(G)$, is the largest number of time steps necessary for a minimum zero forcing set to force the entire vertex set gray.

A non-empty subset $F\subseteq V$ is a \emph{fort} if no vertex $u\in V\setminus{F}$ has exactly one neighbor in $F$.
Let $\mathcal{F}_{G}$ denote the collection of all forts of $G$. 
The \emph{fort number} of $G$, denoted $\ft(G)$, is the maximum number of disjoint forts in $\mathcal{F}_{G}$~\cite{Cameron2023}. 
Moreover, the zero forcing sets of $G$ form a cover for $\mathcal{F}_{G}$; in particular, Theorem~\ref{thm:fort_cover} states that a subset of vertices is a zero forcing set if and only if that subset intersects every fort. 
While one direction of this result was originally proven in~\cite[Theorem 3]{Fast2018}, both directions are shown in~\cite[Theorem 8]{Brimkov2019}.
\begin{theorem}[Theorem 8 in~\cite{Brimkov2019}]\label{thm:fort_cover}
Let $G\in\mathbb{G}$.
Then, $S\subseteq V$ is a zero forcing set of $G$ if and only if $S$ intersects every fort in $\mathcal{F}_{G}$.
\end{theorem}

Theorem~\ref{thm:fort_cover} motivates the fort cover model originally introduced in~\cite{Brimkov2019}, see~\eqref{eq:fc-obj}--\eqref{eq:fc-const2}, where $s_{v}$ is a binary variable that indicates whether $v$ is in the zero forcing set.  
\begin{mini!}
    {}{\sum_{v\in V}s_{v}}{}{}\label{eq:fc-obj}
    \addConstraint{\sum_{v\in F}s_{v}}{\geq 1,~\quad\forall F\in    \mathcal{F}_{G}}\label{eq:fc-const1}
    \addConstraint{s}{\in\{0,1\}^{n}}\label{eq:fc-const2}
\end{mini!}

In~\cite{Cameron2023}, the authors introduce the fractional zero forcing number, denoted $\Z^{*}(G)$, as the optimal value of the linear relaxation of the fort cover model. 
That is, $\Z^{*}(G)$ is the minimum total weight on $V(G)$ such that the sum of weights over each fort is at least $1$. 
Furthermore, they show that the following inequality holds for all graphs:
\begin{equation}\label{eq:fzf_bounds}
\ft(G) \leq \Z^{*}(G) \leq \Z(G).
\end{equation}

\subsection{The hypercube graph}
Let $G=(V,E)$ and $G'=(V',E')$ denote graphs in $\mathbb{G}$ such that $V\cap V'=\emptyset$. 
The \emph{Cartesian product} $G\Box G'$ has vertex set $V(G\Box G') = V(G)\times V(G')$ and edge set
\[
E(G\Box G') = \left\{\{(u,u'),(v,v')\}\colon u=v,~u'\in N_{G'}(v')~\textrm{or}~u'=v',~u\in N_{G}(v)\right\}.
\]
The \emph{hypercube graph} $Q_{d}$ of dimension $d$ can be defined recursively as $Q_{1}=K_{2}$ and
\[
Q_{i} = Q_{i-1}\Box Q_{1},~2\leq i\leq d. 
\]
From this definition, it is clear that $Q_{d}$ is a matching graph with order $n=2^{d}$ and size $m=d\cdot2^{d-1}$, see~\cite[Definition 3.11]{Hogben2012}.
Furthermore, by~\cite[Theorem 3.1]{AIM2008}, $\Z(Q_{d})=2^{d-1}$; hence, by~\cite[Proposition 3.12]{Hogben2012}, $\pt(Q_{d})=1$.

It is known that the hypercube graph can be viewed as the undirected Hasse diagram of a boolean lattice, that is, each vertex corresponds to a binary vector of dimension $d$ and two vertices are adjacent if and only if their binary representations differ by exactly one digit~\cite{Foldes1977}. 
From this representation, the following useful observations are clear.
\begin{observation}\label{obs:dreg}
For all $v\in V(Q_{d})$, $\deg{v}=d$. 
\end{observation}
\begin{observation}\label{obs:bipartite}
The graph $Q_{d}$ is bipartite.
In particular, $V(Q_{d})$ can be split into disjoint sets $V_{e}$ and $V_{o}$, where $V_{e}$ corresponds to binary vectors with an even sum and $V_{o}$ corresponds to binary vectors with an odd sum.
Moreover, every edge of $Q_{d}$ contains a vertex in $V_{e}$ and a vertex in $V_{o}$.
\end{observation}
\begin{observation}\label{obs:cycles}
All cycles in $Q_{d}$ are even. 
\end{observation}
\begin{observation}\label{obs:neighbors}
Let $u\in V(Q_{d})$.
For all $v,w\in N(u)$, $v\notin N(w)$. 
\end{observation}
\section{Minimum forts of the hypercube}\label{sec:minimum_forts}
In this section, we characterize the minimum forts of the hypercube graph. 
As a corollary of this characterization, we derive the fractional zero forcing number as well as lower and upper bounds on the fort number of the hypercube graph. 
Note that the fractional zero forcing number of the hypercube graph is already known, see~\cite[Proposition 3.15]{Cameron2023}.
However, their result relies on the failed zero forcing number of the hypercube graph from~\cite[Theorem 4.2]{Afzali2024}, which does not provide a characterization of all minimum forts but rather identifies a failed zero forcing set of cardinality $(2^{d}-d)$ and argues that every set of larger cardinality is a zero forcing set. 
We begin by noting another result from~\cite{Foldes1977}.
\begin{proposition}[Proposition 2 in~\cite{Foldes1977}]\label{prop:paths}
Let $u,v\in V(Q_{d})$.
Then, there exist $d(u,v)!$ distinct shortest paths between $u$ and $v$. 
\end{proposition}
From Proposition~\ref{prop:paths}, the following proposition is immediate.
\begin{proposition}\label{prop:4cycle}
Let $u,v\in V(Q_{d})$.
If $N(u)\cap N(v)\neq\emptyset$, then $\abs{N(u)\cap N(v)}=2$.
\end{proposition}
\begin{proof}
If $u$ and $v$ are adjacent, then $N(u)\cap N(v)=\emptyset$ since $Q_{d}$ is bipartite by Observation~\ref{obs:bipartite}. 
Therefore, if $N(u)\cap N(v)\neq\emptyset$ then the distance between $u$ and $v$ is $2$ and Proposition~\ref{prop:paths} implies that there exist $2$ distinct paths of length $2$ between $u$ and $v$. 
Hence, $u$ and $v$ have exactly $2$ common neighbors. 
\end{proof}

Note that Corollary~\ref{prop:4cycle} states that any pair of vertices whose neighborhoods intersect are opposite corners of some $4$-cycle in $Q_{d}$. 
Moreover, this corollary implies that every neighborhood in $Q_{d}$ is a fort of $Q_{d}$.
While this result is already known, see~\cite[Proposition 3.15]{Cameron2023}, we provide a proof here for completeness. 
\begin{proposition}\label{prop:nbhd_fort}
Let $d\geq 2$ and $u\in V(Q_{d})$.
Then, $N(u)$ is a fort of $Q_{d}$.
\end{proposition}
\begin{proof}
Let $F=N(u)$.
By Observation~\ref{obs:dreg}, $F$ is non-empty since $\abs{F}=\abs{N(u)}=d$.
Now, let $w\in V(Q_{d})\setminus{F}$.
If $w=u$, then $\abs{N(w)\cap F}=\abs{F}=d$. 
Otherwise, Proposition~\ref{prop:4cycle} implies that $N(w)\cap F=\emptyset$ or $\abs{N(w)\cap F}=2$.
Therefore, $F$ is a fort of $Q_{d}$.
\end{proof}

Next, we show that no minimum fort of $Q_{d}$ contains adjacent vertices. 
It is worth noting that there are minimal forts of $Q_{d}$ that contain adjacent vertices. 
For example, the white vertices shown in Figure~\ref{fig:q3_minimalfort} form a minimal fort of $Q_{3}$.
\begin{figure}[ht]
\centering
\resizebox{0.25\textwidth}{!}{%
\begin{tikzpicture}
[nodeFilledDecorate/.style={shape=circle,inner sep=2pt,draw=black,fill=lightgray,thick},%
    nodeEmptyDecorate/.style={shape=circle,inner sep=2pt,draw=black,thick},%
    lineDecorate/.style={black,=>latex',very thick},%
    scale=2.0]
    \foreach \nodename/\x/\y in {101/0.5/2.5, 010/2/0, 110/2.5/0.5, 001/0/2}
    {
        \node (\nodename) at (\x,\y) [nodeFilledDecorate] {\nodename};
    }
    \foreach \nodename/\x/\y in {000/0/0, 011/2/2, 111/2.5/2.5, 100/0.5/0.5}
    {
        \node (\nodename) at (\x,\y) [nodeEmptyDecorate] {\nodename};
    }
    \path\foreach \startnode/\endnode in {000/010, 010/011, 011/001, 001/000, 100/110, 110/111, 111/101, 101/100, 000/100, 010/110, 011/111, 001/101}
    {
        (\startnode) edge[lineDecorate] node {} (\endnode)
    };
\end{tikzpicture}%
}
\caption{A minimal fort (white) of $Q_{3}$ that contains adjacent vertices.}
\label{fig:q3_minimalfort}
\end{figure}
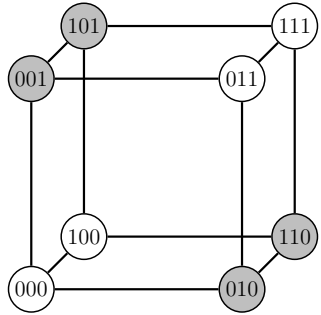
\begin{proposition}\label{prop:minimum_fort_adj}
 Let $d\geq 2$.
 Then, no minimum fort of $Q_{d}$ contains adjacent vertices. 
\end{proposition}
\begin{proof}
Let $F=\{u,v\}\subset V(Q_{d})$, where $u$ and $v$ are adjacent.
Note that $F$ is not a fort of $Q_{d}$.
Indeed, since $u$ and $v$ are adjacent, Observation~\ref{obs:cycles} implies that $N(u)\cap N(v)=\emptyset$ and every $w\in N(u)\cup N(v)$ satisfies $\abs{N(w)\cap F}=1$. 
Hence, by Observation~\ref{obs:dreg}, there are $\abs{N(u)\cup N(v)}=2d-2$ fort violations of $F$. 
There are two options to repair these fort violations by adding a vertex to $F$: Either we add a vertex from $N(u)\cup N(v)$, or we add a vertex that is not in $N(u)\cup N(v)$ but is adjacent to a vertex in $N(u)\cup N(v)$.
Obviously, by adding a vertex to $F$ we may introduce new fort violations, but we will ignore this effect for now. 

By Proposition~\ref{prop:4cycle}, for every $u'\in N(u)$ there is a unique $v'\in N(v)$ such that $u'$ and $v'$ are adjacent. 
By adding $u'$ to $F$ we resolve exactly $2$ fort violations of $F$, namely $u'$ and $v'$.
Similarly, by adding $v'\in N(v)$ to $F$ we resolve exactly $2$ fort violations of $F$. 
Therefore, adding any vertex in $N(u)\cup N(v)$ to $F$ resolves exactly $2$ fort violations of $F$.

Let $p\in V(Q_{d})\setminus\left(N(u)\cup N(v)\right)$.
Since $u$ and $v$ are adjacent, Observation~\ref{obs:cycles} implies that $p$ cannot be adjacent to vertices in both $N(u)$ and $N(v)$.
Hence, by Proposition~\ref{prop:4cycle}, $p$ is either not adjacent to any vertex in $N(u)\cup N(v)$ or is adjacent to exactly  $2$ vertices in $N(u)\cup N(v)$; in the latter case, $p$ is adjacent to either $u',u''\in N(u)$ or $v',v''\in N(v)$. 
Either way, by adding $p$ to $F$ we resolve exactly $2$ fort violations of $F$, namely $u'$ and $u''$ or $v'$ and $v''$. 
Therefore, adding any vertex not in $N(u)\cup N(v)$ to $F$ will either resolve no fort violations or will resolve exactly $2$ fort violations of $F$. 

Since adding vertices to $F$ may introduce new fort violations, it follows that we must add at least $(d-1)$ vertices to resolve all fort violations of $F$.
So, any fort of $Q_{d}$ with adjacent vertices must have cardinality at least $(d+1)$. 
Therefore, by Proposition~\ref{prop:nbhd_fort}, no minimum fort of $Q_{d}$ contains adjacent vertices. 
\end{proof}

Recall from Observation~\ref{obs:bipartite} that $Q_{d}$ is a bipartite graph.
As a consequence of Proposition~\ref{prop:minimum_fort_adj}, we show that no minimum fort of $Q_{d}$ contains vertices from both parts of $Q_{d}$.
\begin{proposition}\label{prop:even_dist}
Let $d\geq 2$ and let $V_{e}$ and $V_{o}$ denote the parts of $Q_{d}$ from Observation~\ref{obs:bipartite}. 
Then, no minimum fort of $Q_{d}$ contains vertices from both $V_{e}$ and $V_{o}$.
\end{proposition}
\begin{proof}
Let $F$ be a fort of $Q_{d}$ that contains vertices from both $V_{e}$ and $V_{o}$. 
If any vertices in $F$ are adjacent, then Proposition~\ref{prop:minimum_fort_adj} says that $F$ is not a minimum fort of $Q_{d}$.
Hence, we may assume that the vertices in $F$ are not adjacent. 

Now, define $F_{e}=F\cap V_{e}$ and $F_{o}=F\cap V_{o}$.
We claim that $F_{e}$ is a fort of $Q_{d}$.
Note that $F_{e}$ is non-empty since $F$ contains vertices in both $V_{e}$ and $V_{o}$.
For the sake of contradiction, suppose that there exists a $u\notin F_{e}$ such that $\abs{N(u)\cap F_{e}}=1$.
Since $Q_{d}$ is bipartite and $F$ has no adjacent vertices, it follows that $u\in V_{o}\setminus{F_{o}}$ and $u$ has no neighbors in $F_{o}$.
However, this implies that $\abs{N(u)\cap F}=1$ which contradicts $F$ being a fort of $Q_{d}$.

Note that a similar argument shows that $F_{o}$ is also a fort of $Q_{d}$. 
Since $F_{e}\subset F$ is a fort of $Q_{d}$, it follows that $F$ is not a minimum fort of $Q_{d}$.
Hence, no minimum fort of $Q_{d}$ contains vertices from both $V_{e}$ and $V_{o}$.
\end{proof}

Proposition~\ref{prop:minimum_fort_adj} implies that the minimum forts of $Q_{d}$ contain only vertices that are an even distance apart from one another. 
The following proposition shows that the maximum distance between any pair of vertices in a minimum fort of $Q_{d}$ is less than $6$. 
\begin{proposition}\label{prop:max_dist6}
Let $d\geq 6$ and let $u,v\in V(Q_{d})$ such that $\dist{u}{v}\geq 6$.
Then, there is no minimum fort of $Q_{d}$ that contains both $u$ and $v$. 
\end{proposition}
\begin{proof}
Let $F=\{u,v\}$.
Note that $F$ is not a fort of $Q_{d}$ since every $w\in N(u)\cup N(v)$ satisfies $\abs{N(w)\cap F}=1$. 
Hence, by Observation~\ref{obs:dreg}, there are $\abs{N(u)\cup N(v)}=2d$ fort violations of $F$.
There are two options to repair these fort violations by adding a vertex to $F$: Either we add a vertex from $N(u)\cup N(v)$, or we add a vertex that is not in $N[u]\cup N[v]$ but is adjacent to a vertex in $N[u]\cup N[v]$.
Obviously, by adding a vertex to $F$ we may introduce new fort violations, but we will ignore this effect for now.

Let $u'\in N(u)$.
By Observation~\ref{obs:neighbors}, $u'$ is not adjacent to any vertex in $N(u)$.
Furthermore, since $\dist{u}{v}\geq 6$, $u'$ is not adjacent to any vertex in $N(v)$. 
Therefore, adding $u'$ to $F$ resolves exactly $1$ fort violation of $F$, namely $u'$. 
Similarly, adding $v'\in N(v)$ to $F$ resolves exactly $1$ fort violation of $F$, namely $v'$. 

Let $p\in V(Q_{d})\setminus\left(N[u]\cup N[v]\right)$. 
Since $\dist{u}{v}\geq 6$, it follows that $p$ cannot be adjacent to vertices in both $N(u)$ and $N(v)$.
Hence, by Proposition~\ref{prop:4cycle}, $p$ is either not adjacent to any vertex in $N(u)\cup N(v)$ or is adjacent to exactly $2$ vertices in $N(u)\cup N(v)$; in the latter case, $p$ is adjacent to either $u',u''\in N(u)$ or $v',v''\in N(v)$. 
Either way, by adding $p$ to $F$ we resolve exactly $2$ fort violations of $F$, namely $u'$ and $u''$ or $v'$ and $v''$. 
Therefore, adding any vertex not in $N[u]\cup N[v]$ to $F$ will either not resolve any fort violations or will resolve exactly $2$ fort violations of $F$.

Since adding vertices to $F$ may introduce new fort violations, it follows that we must add at least $d$ vertices to resolve all fort violations of $F$.
So, any fort of $Q_{d}$ that contains $u$ and $v$ must have cardinality at least $(d+2)$.
Therefore, by Proposition~\ref{prop:nbhd_fort}, no minimum fort of $Q_{d}$ contains vertices a distance of at least $6$ apart.
\end{proof}

Next, we show that any fort of $Q_{d}$ that contains vertices a distance of $4$ apart must contain at least $d$ vertices.
Moreover, if $d\geq 5$, then any fort containing vertices a distance of $4$ apart must contain at least $(d+1)$ vertices. 
Hence, $Q_{4}$ is the only hypercube graph that has minimum forts with vertices a distance of $4$ apart, see Figure~\ref{fig:q4_minfort}.
\begin{proposition}\label{prop:max_dist4}
Let $d\geq 4$ and let $u,v\in V(Q_{d})$ such that $\dist{u}{v}=4$.
If $d=4$, then any fort $F$ containing $u$ and $v$ satisfies $\abs{F}\geq d$.
If $d\geq 5$, then any fort $F$ containing $u$ and $v$ satisfies $\abs{F}>d$. 
\end{proposition}
\begin{proof}
Let $F=\{u,v\}\subset V(Q_{d})$.
Note that $F$ is not a fort of $Q_{d}$ since every $w\in N(u)\cup N(v)$ satisfies $\abs{N(w)\cap F}=1$. 
Hence, by Observation~\ref{obs:dreg}, there are $\abs{N(u)\cup N(v)}=2d$ fort violations of $F$.
There are two options to repair these fort violations by adding a vertex to $F$: Either we add a vertex from $N(u)\cup N(v)$, or we add a vertex that is not in $N[u]\cup N[v]$ but is adjacent to a vertex in $N[u]\cup N[v]$.
Obviously, by adding a vertex to $F$ we may introduce new fort violations, but we will ignore this effect for now.

Let $u'\in N(u)$.
By Observation~\ref{obs:neighbors}, $u'$ is not adjacent to any vertex in $N(u)$.
Furthermore, since $\dist{u}{v}=4$, $u'$ is not adjacent to any vertex in $N(v)$. 
Therefore, adding $u'$ to $F$ resolves exactly $1$ fort violation of $F$, namely $u'$.
Similarly, adding $v'\in N(v)$ to $F$ resolves exactly $1$ fort violation of $F$, namely $v'$. 

Let $p\in V(Q_{d})\setminus\left(N[u]\cup N[v]\right)$.
Then, by Proposition~\ref{prop:4cycle}, $p$ may be adjacent to no vertices in $N(u)\cup N(v)$, exactly $2$ vertices in $N(u)$ or $2$ vertices in $N(v)$, or exactly $2$ vertices in $N(u)$ and $2$ vertices in $N(v)$. 
In the latter case, $p$ is adjacent to $u',u''\in N(u)$ and $v',v''\in N(v)$ such that $d(u',v)=d(u'',v)=3$ and $d(v',u)=d(v'',u)=3$.
There are exactly $4$ vertices in $N(u)$ a distance of $3$ away from $v$; similarly, there are exactly $4$ vertices in $N(v)$ a distance of $3$ away from $u$.
Moreover, all $8$ of these fort violations can be resolved by adding $p_{1},p_{2}\in V(Q_{d})\setminus\left(N[u]\cup N[v]\right)$ to $F$, where $p_{1}$ and $p_{2}$ are a distance of $2$ away from both $u$ and $v$ and $\dist{p_{1}}{p_{2}}=4$.
After adding $p_{1}$ and $p_{2}$ to $F$, we still have $2d-8$ fort violations from $N(u)\cup N(v)$.
Furthermore, every vertex that can be added at this point resolves at most $2$ of these fort violations.
Hence, we must add at least $d-4$ more vertices to $F$ in order to resolve all fort violations.

Suppose that $d=5$.
Then, there are exactly two fort violations that remain from $N(u)\cup N(v)$, which we denote by $u'\in N(u)$ and $v'\in N(v)$.
Since $d(u',v)=d(v',u)=5$, these fort violations cannot be resolved by a single vertex. 
Hence, when $d=5$, at least $4$ vertices must be added to $F$ in order resolve all fort violations. 

Finally, suppose that $d\geq 6$. 
Then, there are at least two fort violations that remain from $N(u)$, which we denote by $u',u''$.
Note that $d(u',v)=d(u'',v)=5$. 
Moreover, by Proposition~\ref{prop:4cycle}, there exists a unique $w$ such that $N(w)\cap N(u)=\{u',u''\}$; hence, we can resolve two fort violations with $w$.
However, $d(w,v)=6$ and Proposition~\ref{prop:max_dist6} implies that no minimum fort of $Q_{d}$ will contain both $w$ and $v$.
\end{proof}

In Figure~\ref{fig:q4_minfort}, we show a minimum fort of $Q_{4}$ that is not the neighborhood of a single vertex.
By Propositions~\ref{prop:even_dist},~\ref{prop:max_dist6}, and~\ref{prop:max_dist4}, there are no such minimum forts of $Q_{d}$ for $d\geq 2$, $d\neq 4$.
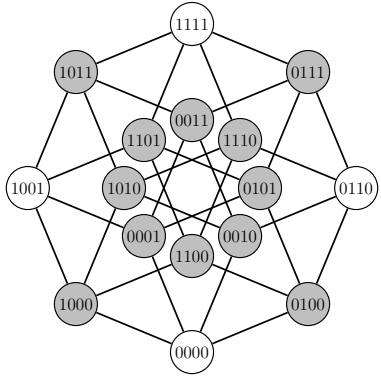
\begin{figure}[ht]
\centering
\resizebox{0.3\textwidth}{!}{%
\begin{tikzpicture}
[nodeFilledDecorate/.style={shape=circle,inner sep=2pt,draw=black,fill=lightgray,thick},%
    nodeEmptyDecorate/.style={shape=circle,inner sep=2pt,draw=black,thick},%
    lineDecorate/.style={black,=>latex',very thick},%
    scale=4.0]
    \foreach \nodename/\x/\y in {0111/0.7071/0.7071, 0100/0.7071/-0.7071, 0101/0.4142/0, 0011/0/0.4142, 0001/-0.2928/-0.2928, 0010/0.2928/-0.2928, 1011/-0.7071/0.7071, 1000/-0.7071/-0.7071, 1110/0.2928/0.2928, 1101/-0.2928/0.2928, 1010/-0.4142/0, 1100/0/-0.4142}
    {
        \node (\nodename) at (\x,\y) [nodeFilledDecorate] {\nodename};
    }
    \foreach \nodename/\x/\y in {0000/0/-1, 1111/0/1, 0110/1/0, 1001/-1/0}
    {
        \node (\nodename) at (\x,\y) [nodeEmptyDecorate] {\nodename};
    }
    \path\foreach \startnode/\endnode in {0110/0111, 0110/0100, 0110/1110, 0110/0010, 0111/1111, 0111/0101, 0111/0011, 1111/1011, 1111/1110, 1111/1101, 1011/1001, 1011/0011, 1011/1010, 1001/1000, 1001/1101, 1001/0001, 1000/0000, 1000/1010, 1000/1100, 0000/0100, 0000/0001, 0000/0010, 0100/0101, 0100/1100, 0101/1101, 0101/0001, 1110/1010, 1110/1100, 0011/0001, 0011/0010, 1101/1100, 1010/0010}
    {
        (\startnode) edge[lineDecorate] node {} (\endnode)
    };
\end{tikzpicture}%
}
\caption{A minimum fort (white) of $Q_{4}$ that is not the neighborhood of a single vertex.}
\label{fig:q4_minfort}
\end{figure}

The only case left to consider is where all vertices in the fort are a distance of $2$ apart.
Of course, this fort structure is only achieved when all vertices are members of the neighborhood of a central vertex. 
Therefore, we have the following result.
\begin{theorem}\label{thm:min_fort}
Let $d\geq 2$, $d\neq 4$, and $F\subseteq V(Q_{d})$.
Then, $F$ is a minimum fort of $Q_{d}$ if and only if $F=N(v)$ for some $v\in Q_{d}$. 
\end{theorem}
\begin{proof}
Suppose that $F$ is a minimum fort of $Q_{d}$.
Then, by Proposition~\ref{prop:even_dist}, $F$ only contains vertices that are an even distance apart from each other. 
By Proposition~\ref{prop:max_dist6}, $F$ does not contain vertices that are a distance of $6$ or more apart. 
If $F$ contains vertices of distance $4$ apart, then by Proposition~\ref{prop:max_dist4}, $\abs{F}>d$ since $d\neq 4$. 
However, if $\abs{F}>d$, then $F$ cannot be a minimum fort since all the neighborhoods of $Q_{d}$ are forts by Proposition~\ref{prop:nbhd_fort}. 
Therefore, all vertices in $F$ are a distance of $2$ apart from each other, which means $F$ is the neighborhood of a single vertex. 

Conversely, suppose that $F=N(v)$ for some $v\in V(Q_{d})$.
Then, by Proposition~\ref{prop:nbhd_fort}, $F$ is fort of $Q_{d}$.
Moreover, by Propositions~\ref{prop:even_dist},~\ref{prop:max_dist6}, and~\ref{prop:max_dist4}, it follows that $F$ is a minimum fort of $Q_{d}$.
\end{proof}

The following proposition characterizes the minimum forts on $Q_{4}$.
\begin{proposition}\label{prop:min_fort4}
Let $F\subseteq V(Q_{4})$. 
Then, $F$ is a minimum fort if and only if $F=N(v)$ for some $v\in Q_{4}$ or there is an automorphism of $Q_{4}$ that maps $F$ to the set shown in Figure~\ref{fig:q4_minfort}.
\end{proposition}
\begin{proof}
Suppose that $F$ is a minimum fort of $Q_{4}$.
Then, by Proposition~\ref{prop:even_dist}, $F$ only contains vertices that are an even distance apart from each other. 
If $F$ contains vertices a distance of $4$ apart, then by Proposition~\ref{prop:max_dist4}, $\abs{F}\geq 4$.
Since $F$ is a minimum fort, Proposition~\ref{prop:nbhd_fort} implies that $\abs{F}=4$.
Let $u,v\in F$ such that $\dist{u}{v}=4$.
In order to resolve all fort violations in $N(u)\cup N(v)$, there must exist $u',v'\in F$ such that $\dist{u'}{v'}=4$, and $u'$ and $v'$ are a distance of $2$ away from both $u$ and $v$. 
All sets $F$ of this form are automorphic to the set in Figure~\ref{fig:q4_minfort}.
If all vertices of $F$ are a distance of $2$ apart, then $F=N(v)$ for some $v\in V(Q_{4})$.

Conversely, suppose that $F=N(v)$ for some $v\in V(Q_{4})$ or there is an automorphism of $Q_{4}$ that maps $F$ to the set shown in Figure~\ref{fig:q4_minfort}.
In either case, $F$ is clearly a fort of $Q_{4}$.
Moreover, by Propositions~\ref{prop:even_dist} and~\ref{prop:max_dist4}, it follows that $F$ is a minimum fort of $Q_{4}$.
\end{proof}

We now state several implications of Theorem~\ref{thm:min_fort} and Proposition~\ref{prop:min_fort4}.
\begin{corollary}\label{cor:min_fort}
Let $d\geq 2$ and let $F$ be a minimum fort of $Q_{d}$.
Then, $\abs{F}=d$.
\end{corollary}
\begin{proof}
If $d\neq 4$,
then Theorem~\ref{thm:min_fort} implies that $F=N(v)$ for some $v\in V(Q_{d})$.
Hence, by Observation~\ref{obs:dreg}, $\abs{F}=d$.
Otherwise, if $d=4$, then Proposition~\ref{prop:min_fort4} implies that $\abs{F}=d$. 
\end{proof}
\begin{corollary}\label{cor:fzf_hypercube}
For $d\geq 2$, $\Z^{*}(Q_{d})=\frac{2^{d}}{d}$.
\end{corollary}
\begin{proof}
By Proposition~\ref{prop:nbhd_fort} and Corollary~\ref{cor:min_fort}, the neighborhoods of $Q_{d}$ are minimum forts.
Moreover, by Corollary~\ref{cor:min_fort}, there is no fort of $Q_{d}$ with cardinality less than $d$.
Therefore, we must weight each vertex by $1/d$ so that the sum of weights over each fort is at least $1$.
Since no weighting of less than $1/d$ on a vertex will suffice, it follows that $\Z^{*}(Q_{d})=\frac{2^{d}}{d}$.
\end{proof}
\begin{corollary}\label{cor:ft_bounds}
Let $d\geq 2$.
Then, 
\[
\frac{2^{d}}{2^{\floor{\log_{2}(d-1)}+1}} \leq \ft(Q_{d}) \leq \frac{2^{d}}{d}. 
\]
\end{corollary}
\begin{proof}
By Corollary~\ref{cor:fzf_hypercube} and the bounds in~\eqref{eq:fzf_bounds}, it follows that $\ft(Q_{d}) \leq \frac{2^{d}}{d}$.
Furthermore, by Proposition~\ref{prop:nbhd_fort}, each neighborhood in $Q_{d}$ is a fort, so $\ft(Q_{d})\geq\rho^{o}(Q_{d})$, where $\rho^{o}(G)$ is the open packing number of $G$, that is, the cardinality of the largest set of vertices in $G$ no two neighborhoods of which intersect. 
By~\cite[Theorem 3.2]{Bresar2024}, $\rho^{o}(Q_{d}) \geq 2^{d-\floor{\log_{2}(d-1)}-1}$.
Therefore, $\frac{2^{d}}{2^{\floor{\log_{2}(d-1)}+1}} \leq \ft(Q_{d}) \leq \frac{2^{d}}{d}$.
\end{proof}

We now show that the lower and upper bounds from Corollary~\ref{cor:ft_bounds} are equal for all $d=2^{k}$ for some $k\in\mathbb{N}$.
\begin{theorem}\label{thm:ft_sharp_bound}
Let $k\in\mathbb{N}$.
Then, $\ft(Q_{2^{k}}) = \Z^{*}(Q_{2^{k}}) = 2^{2^{k}-k}$.
\end{theorem}
\begin{proof}
Suppose that $d=2^{k}$ for some $k\in\mathbb{N}$. 
Then, $\Z^{*}(Q_{d}) = \frac{2^{d}}{d}=2^{2^{k}-k}$. 
Moreover, 
\[
\log_{2}(d-1) = \log_{2}(2^{k}-1) = \log_{2}(2^{k}) + \log_{2}(1-2^{-k}).
\]
Note that $-1 \leq \log_{2}(1-2^{-k}) < 0$; hence,
\[
k-1 \leq \log_{2}(2^{k}) + \log_{2}(1-2^{-k}) < k,
\]
which implies that $\floor{\log_{2}(2^{k}-1)} = k-1$.
Therefore, 
\[
\frac{2^{d}}{2^{\floor{\log_{2}(d-1)}+1}} = 2^{2^{k}-k} = \frac{2^{d}}{d},
\]
and the result follows from Corollary~\ref{cor:ft_bounds}.
\end{proof}

 A dominating set of $G$ is a set of vertices such that any vertex of $G$ is in the set or has a neighbor in the set.
The domination number $\gamma(G)$ is the minimum cardinality of a dominating set of $G$. 
A total dominating set of $G$ is a set of vertices such that any vertex of $G$, including those in the set, have a neighbor in the set.
The total domination number $\gamma_{t}(G)$ is the minimum cardinality of a dominating set of $G$. In~\cite{Bresar2024}, it was shown that for $k\in\mathbb{N}$,
$\rho^{o}(Q_{2^{k}}) = \gamma(Q_{2^{k}}) = \gamma_{t}(Q_{2^{k}}) = 2^{2^{k}-k}$.
As a consequence of Theorem~\ref{thm:ft_sharp_bound}, when the dimension of the hypercube graph is a power of $2$ we get equality among $5$ distinct graph parameters.
\begin{corollary}\label{cor:sharp_bounds}
Let $k\in\mathbb{N}$.
Then,
\[
\ft(Q_{2^{k}}) = \Z^{*}(Q_{2^{k}}) = \rho^{o}(Q_{2^{k}}) = \gamma(Q_{2^{k}}) = \gamma_{t}(Q_{2^{k}}) = 2^{2^{k}-k}. 
\]
\end{corollary}
\section{Minimum zero forcing sets of the hypercube}\label{sec:minimum_zf_sets}
By Theorem~\ref{thm:min_fort}, for $d\neq 4$, the minimum forts of $Q_{d}$ are all automorphic to each other. Thus, the maximum failed zero forcing sets of $Q_{d}$ are also automorphic to each other. 
In contrast, we show in this  section that not all minimum zero forcing sets of $Q_d$ are automorphic. We also conjecture that there are non-automorphic minimum zero forcing sets in a particularly strong sense.

First, we note that the only minimum zero forcing sets of $Q_{2}$ are adjacent pairs of vertices, and each of these sets has a propagation time of $1$.
In $Q_{3}$, there are minimum zero forcing sets with propagation time $1$ and $2$, see Figure~\ref{fig:q3_pt1-2}.
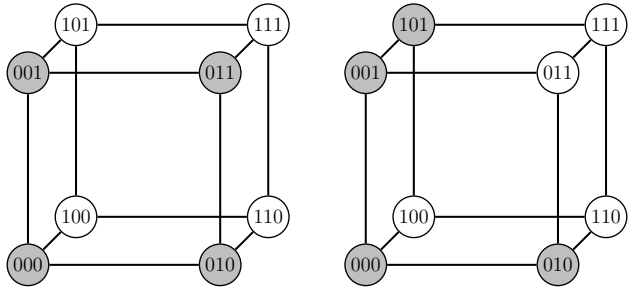
\begin{figure}[ht]
\centering
\resizebox{0.5\textwidth}{!}{%
\begin{tikzpicture}
[nodeFilledDecorate/.style={shape=circle,inner sep=2pt,draw=black,fill=lightgray,thick},%
    nodeEmptyDecorate/.style={shape=circle,inner sep=2pt,draw=black,thick},%
    lineDecorate/.style={black,=>latex',very thick},%
    scale=2.0]
    \begin{scope}
    \foreach \nodename/\x/\y in {000/0/0, 010/2/0, 011/2/2, 001/0/2}
    {
        \node (\nodename) at (\x,\y) [nodeFilledDecorate] {\nodename};
    }
    \foreach \nodename/\x/\y in {100/0.5/0.5, 110/2.5/0.5, 111/2.5/2.5, 101/0.5/2.5}
    {
        \node (\nodename) at (\x,\y) [nodeEmptyDecorate] {\nodename};
    }
    \path\foreach \startnode/\endnode in {000/010, 010/011, 011/001, 001/000, 100/110, 110/111, 111/101, 101/100, 000/100, 010/110, 011/111, 001/101}
    {
        (\startnode) edge[lineDecorate] node {} (\endnode)
    };
    \end{scope}
    \begin{scope}[xshift=100]
    \foreach \nodename/\x/\y in {000/0/0, 010/2/0, 101/0.5/2.5, 001/0/2}
    {
        \node (\nodename) at (\x,\y) [nodeFilledDecorate] {\nodename};
    }
    \foreach \nodename/\x/\y in {100/0.5/0.5, 110/2.5/0.5, 111/2.5/2.5, 011/2/2}
    {
        \node (\nodename) at (\x,\y) [nodeEmptyDecorate] {\nodename};
    }
    \path\foreach \startnode/\endnode in {000/010, 010/011, 011/001, 001/000, 100/110, 110/111, 111/101, 101/100, 000/100, 010/110, 011/111, 001/101}
    {
        (\startnode) edge[lineDecorate] node {} (\endnode)
    };
    \end{scope}
\end{tikzpicture}%
}
\caption{Minimum zero forcing sets of $Q_{3}$ shown in gray, with propagation time $1$ (left) and propagation time $2$ (right).}
\label{fig:q3_pt1-2}
\end{figure}

The following proposition shows how to construct minimum zero forcing sets of $Q_{d+1}$ that have the same propagation time as a minimum zero forcing set of $Q_{d}$.
\begin{proposition}\label{prop:prop_copy}
Let $S$ be a minimum zero forcing set of $Q_{d}$ with propagation time $k$. 
Then, there exists a minimum zero forcing set of $Q_{d+1}$ with propagation time $k$. 
\end{proposition}

\begin{proof}
Let $G$ be a hypercube of dimension $d$ with $V(G)=\{v_1,\ldots,v_n\}$, and let $S=\{v_{i_1},\ldots,v_{i_t}\}$ be a minimum zero forcing set of $G$ with propagation time $k$. Let $G'$ be another hypercube with dimension $d$ with $V(G')=\{v_1',\ldots,v_n'\}$ such that an isomorphism between $G$ and $G'$ maps $v_i$ to $v_i'$ for $1\leq i\leq n$. Let $S'=\{v_{i_1}',\ldots,v_{i_t}'\}$. Then, $S'$ is a minimum zero forcing set of $G'$ with propagation time $k$. Consider the graph $H$ obtained by starting from the disjoint union of $G$ and $G'$ and adding edges $\{v_i,v_i'\}$ for $1\leq i\leq n$. Then, $H$ is a hypercube of dimension $d+1$. We claim that $\hat{S}=S\cup S'$ is a zero forcing set of $H$. Let $S_1$ be the set of vertices in $G$ that can force in the first timestep, and let $v$ be a vertex in $S_1$. Then, $v$ has a single white neighbor in $G$. In $H$, $v$ has the same neighbors as in $G$, except it is also adjacent to $v'$ in $G'$, but $v'$ is in $S'$. Thus, $v$ again has a single white neighbor in $H$ and can force that neighbor in the first timestep. The same reasoning can be applied inductively to $S_i$, the set of vertices in $G$ that can force in the $i^{\text{th}}$ timestep. Thus, the propagation time of $\hat{S}$ is equal to the propagation time of $S$. Moreover, $\hat{S}$ must be a minimum zero forcing set of $H$, since $|\hat{S}|=2|S|$ and $Z(Q_{d+1})=2Z(Q_d)$. 
\end{proof}

Using Proposition~\ref{prop:prop_copy}, we can construct minimum zero forcings sets of $Q_{4}$ with propagation time $1$ and $2$ using copies of minimum zero forcing sets of $Q_{3}$ with propagation time $1$ and $2$, respectively; we display these minimum zero forcing sets in Figure~\ref{fig:q4_pt1-2}.
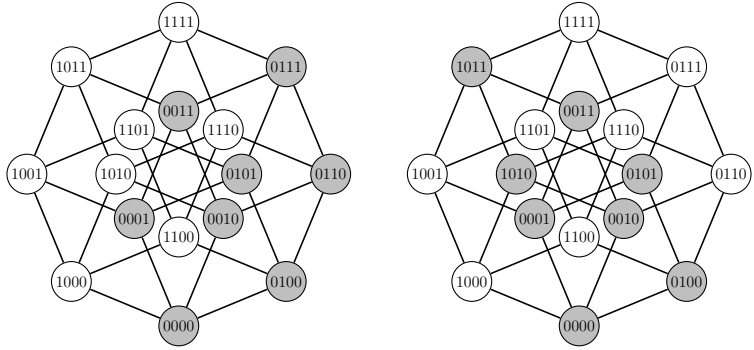
\begin{figure}[ht]
\centering
\resizebox{0.6\textwidth}{!}{%
\begin{tikzpicture}
[nodeFilledDecorate/.style={shape=circle,inner sep=2pt,draw=black,fill=lightgray,thick},%
    nodeEmptyDecorate/.style={shape=circle,inner sep=2pt,draw=black,thick},%
    lineDecorate/.style={black,=>latex',very thick},%
    scale=4.0]
    \begin{scope}
    \foreach \nodename/\x/\y in {0110/1/0, 0111/0.7071/0.7071, 0000/0/-1, 0100/0.7071/-0.7071, 0101/0.4142/0, 0011/0/0.4142, 0001/-0.2928/-0.2928, 0010/0.2928/-0.2928}
    {
        \node (\nodename) at (\x,\y) [nodeFilledDecorate] {\nodename};
    }
    \foreach \nodename/\x/\y in {1111/0/1, 1011/-0.7071/0.7071, 1001/-1/0, 1000/-0.7071/-0.7071, 1110/0.2928/0.2928, 1101/-0.2928/0.2928, 1010/-0.4142/0, 1100/0/-0.4142}
    {
        \node (\nodename) at (\x,\y) [nodeEmptyDecorate] {\nodename};
    }
    \path\foreach \startnode/\endnode in {0110/0111, 0110/0100, 0110/1110, 0110/0010, 0111/1111, 0111/0101, 0111/0011, 1111/1011, 1111/1110, 1111/1101, 1011/1001, 1011/0011, 1011/1010, 1001/1000, 1001/1101, 1001/0001, 1000/0000, 1000/1010, 1000/1100, 0000/0100, 0000/0001, 0000/0010, 0100/0101, 0100/1100, 0101/1101, 0101/0001, 1110/1010, 1110/1100, 0011/0001, 0011/0010, 1101/1100, 1010/0010}
    {
        (\startnode) edge[lineDecorate] node {} (\endnode)
    };
    \end{scope}
    \begin{scope}[xshift=75]
    \foreach \nodename/\x/\y in {1011/-0.7071/0.7071, 1010/-0.4142/0, 0000/0/-1, 0100/0.7071/-0.7071, 0101/0.4142/0, 0011/0/0.4142, 0001/-0.2928/-0.2928, 0010/0.2928/-0.2928}
    {
        \node (\nodename) at (\x,\y) [nodeFilledDecorate] {\nodename};
    }
    \foreach \nodename/\x/\y in {0111/0.7071/0.7071, 0110/1/0, 1111/0/1, 1001/-1/0, 1000/-0.7071/-0.7071, 1110/0.2928/0.2928, 1101/-0.2928/0.2928, 1100/0/-0.4142}
    {
        \node (\nodename) at (\x,\y) [nodeEmptyDecorate] {\nodename};
    }
    \path\foreach \startnode/\endnode in {0110/0111, 0110/0100, 0110/1110, 0110/0010, 0111/1111, 0111/0101, 0111/0011, 1111/1011, 1111/1110, 1111/1101, 1011/1001, 1011/0011, 1011/1010, 1001/1000, 1001/1101, 1001/0001, 1000/0000, 1000/1010, 1000/1100, 0000/0100, 0000/0001, 0000/0010, 0100/0101, 0100/1100, 0101/1101, 0101/0001, 1110/1010, 1110/1100, 0011/0001, 0011/0010, 1101/1100, 1010/0010}
    {
        (\startnode) edge[lineDecorate] node {} (\endnode)
    };
    \end{scope}
\end{tikzpicture}%
}
\caption{Minimum zero forcing sets of $Q_{4}$ shown in gray, with propagation time $1$ (left) and propagation time $2$ (right).}
\label{fig:q4_pt1-2}
\end{figure} 
\begin{figure}[ht]
\centering
\resizebox{0.6\textwidth}{!}{%
\begin{tikzpicture}
[nodeFilledDecorate/.style={shape=circle,inner sep=2pt,draw=black,fill=lightgray,thick},%
    nodeEmptyDecorate/.style={shape=circle,inner sep=2pt,draw=black,thick},%
    lineDecorate/.style={black,=>latex',very thick},%
    scale=4.0]
    \begin{scope}
    \foreach \nodename/\x/\y in {0011/0/0.4142, 1101/-0.2928/0.2928, 1001/-1/0, 0001/-0.2928/-0.2928, 1010/-0.4142/0, 0000/0/-1, 0100/0.7071/-0.7071, 0010/0.2928/-0.2928}
    {
        \node (\nodename) at (\x,\y) [nodeFilledDecorate] {\nodename};
    }
    \foreach \nodename/\x/\y in {0101/0.4142/0, 1011/-0.7071/0.7071, 0111/0.7071/0.7071, 1111/0/1, 1000/-0.7071/-0.7071, 1110/0.2928/0.2928, 1100/0/-0.4142, 0110/1/0}
    {
        \node (\nodename) at (\x,\y) [nodeEmptyDecorate] {\nodename};
    }
    \path\foreach \startnode/\endnode in {0110/0111, 0110/0100, 0110/1110, 0110/0010, 0111/1111, 0111/0101, 0111/0011, 1111/1011, 1111/1110, 1111/1101, 1011/1001, 1011/0011, 1011/1010, 1001/1000, 1001/1101, 1001/0001, 1000/0000, 1000/1010, 1000/1100, 0000/0100, 0000/0001, 0000/0010, 0100/0101, 0100/1100, 0101/1101, 0101/0001, 1110/1010, 1110/1100, 0011/0001, 0011/0010, 1101/1100, 1010/0010}
    {
        (\startnode) edge[lineDecorate] node {} (\endnode)
    };
    \end{scope}
    \begin{scope}[xshift=75]
    \foreach \nodename/\x/\y in {1111/0/1, 1101/-0.2928/0.2928, 1001/-1/0, 0001/-0.2928/-0.2928, 1010/-0.4142/0, 0000/0/-1, 0100/0.7071/-0.7071, 0010/0.2928/-0.2928}
    {
        \node (\nodename) at (\x,\y) [nodeFilledDecorate] {\nodename};
    }
    \foreach \nodename/\x/\y in {0011/0/0.4142, 0101/0.4142/0, 1011/-0.7071/0.7071, 0111/0.7071/0.7071, 1000/-0.7071/-0.7071, 1110/0.2928/0.2928, 1100/0/-0.4142, 0110/1/0}
    {
        \node (\nodename) at (\x,\y) [nodeEmptyDecorate] {\nodename};
    }
    \path\foreach \startnode/\endnode in {0110/0111, 0110/0100, 0110/1110, 0110/0010, 0111/1111, 0111/0101, 0111/0011, 1111/1011, 1111/1110, 1111/1101, 1011/1001, 1011/0011, 1011/1010, 1001/1000, 1001/1101, 1001/0001, 1000/0000, 1000/1010, 1000/1100, 0000/0100, 0000/0001, 0000/0010, 0100/0101, 0100/1100, 0101/1101, 0101/0001, 1110/1010, 1110/1100, 0011/0001, 0011/0010, 1101/1100, 1010/0010}
    {
        (\startnode) edge[lineDecorate] node {} (\endnode)
    };
    \end{scope}
\end{tikzpicture}%
}
\caption{Minimum zero forcing sets of $Q_{4}$ shown in gray, with propagation time $3$ (left) and propagation time $4$ (right).}
\label{fig:q4_pt3-4}
\end{figure}
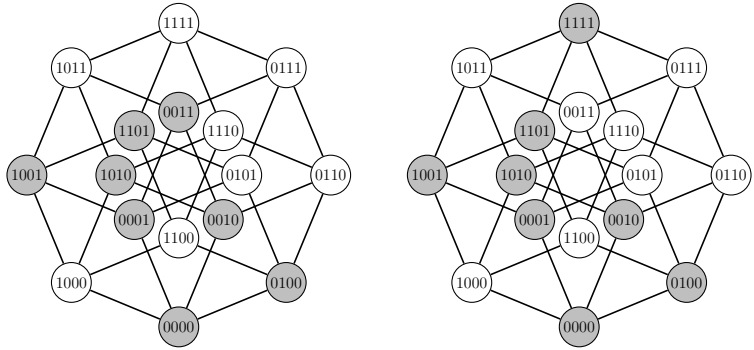

In addition, $Q_{4}$ has minimum zero forcing sets with propagation time $3$ and $4$, see Figure~\ref{fig:q4_pt3-4}.
From Figures~\ref{fig:q3_pt1-2}--\ref{fig:q4_pt3-4} and Proposition~\ref{prop:prop_copy} we can make the following conclusions.
\begin{corollary}\label{cor:pt1-4}
For $d\geq 4$, there exist minimum zero forcing sets of $Q_{d}$ with propagation time $1$, $2$, $3$, and $4$. 
\end{corollary}

Since minimum zero forcing sets with distinct propagation times cannot be automorphic, we obtain the following corollary.
\begin{corollary}\label{cor:non_auto_zf}
For $d\geq 3$, there exist non-automorphic minimum zero forcing sets of $Q_{d}$. 
\end{corollary}

We conclude this section with the following conjecture regarding the propagation time interval of the hypercube graph.
\begin{conjecture}\label{con:full_pt_int}
Let $d\geq 2$.
Then, $\pt(Q_{d}) = 1$ and $\PT(Q_{d})=2^{d-2}$.
Moreover, $Q_{d}$ has a full propagation time interval. 
\end{conjecture}
\section{Minimal Forts of Cartesian Products}\label{sec:minimal_forts}
In this section, we derive constructions for minimal forts of a Cartesian product of graphs.
In the process, we show that these constructions can account for some, but not all, of the minimal forts of the hypercube graph.

We begin with a construction of minimal forts of $G\Box G'$, where $G$ and $G'$ are any graphs. 
It is known that if $F$ is a fort of $G$ and $F'$ is a fort of $G'$, then $F\times F'$ is a fort of $G\Box G'$, see~\cite[Proposition 4.1]{Cameron2023}.
In Theorem~\ref{thm:cart_prod}, we show that if $F$ and $F'$ are minimal and $F$ contains no adjacent vertices, then $F\times F'$ is minimal. 

It is worth noting the necessity of $F$ containing no adjacent vertices.
Indeed, let $F'=V(Q_{1})$ denote the only fort of $Q_{1}$ and let $F$ denote the minimal fort of $Q_{3}$ shown in Figure~\ref{fig:q3_minimalfort}. 
Then, the fort $F\times F'$ is shown in Figure~\ref{fig:q4_fort}, which is not minimal since it is a superset of the fort of $Q_{4}$ shown in Figure~\ref{fig:q4_minfort}.
\begin{figure}[ht]
\centering
\resizebox{0.3\textwidth}{!}{%
\begin{tikzpicture}
[nodeFilledDecorate/.style={shape=circle,inner sep=2pt,draw=black,fill=lightgray,thick},%
    nodeEmptyDecorate/.style={shape=circle,inner sep=2pt,draw=black,thick},%
    lineDecorate/.style={black,=>latex',very thick},%
    scale=4.0]
    \foreach \nodename/\x/\y in {0100/0.7071/-0.7071, 0101/0.4142/0, 0011/0/0.4142, 0010/0.2928/-0.2928, 1011/-0.7071/0.7071, 1101/-0.2928/0.2928, 1010/-0.4142/0, 1100/0/-0.4142}
    {
        \node (\nodename) at (\x,\y) [nodeFilledDecorate] {\nodename};
    }
    \foreach \nodename/\x/\y in {0001/-0.2928/-0.2928, 1110/0.2928/0.2928, 1000/-0.7071/-0.7071, 0111/0.7071/0.7071, 0000/0/-1, 1111/0/1, 0110/1/0, 1001/-1/0}
    {
        \node (\nodename) at (\x,\y) [nodeEmptyDecorate] {\nodename};
    }
    \path\foreach \startnode/\endnode in {0110/0111, 0110/0100, 0110/1110, 0110/0010, 0111/1111, 0111/0101, 0111/0011, 1111/1011, 1111/1110, 1111/1101, 1011/1001, 1011/0011, 1011/1010, 1001/1000, 1001/1101, 1001/0001, 1000/0000, 1000/1010, 1000/1100, 0000/0100, 0000/0001, 0000/0010, 0100/0101, 0100/1100, 0101/1101, 0101/0001, 1110/1010, 1110/1100, 0011/0001, 0011/0010, 1101/1100, 1010/0010}
    {
        (\startnode) edge[lineDecorate] node {} (\endnode)
    };
\end{tikzpicture}%
}
\caption{A non-minimal fort (white) of $Q_{4}$ constructed from $F\times F'$, where $F'=V(Q_{1})$ and $F$ is the minimal fort of $Q_{3}$ shown in Figure~\ref{fig:q3_minimalfort}.}
\label{fig:q4_fort}
\end{figure}
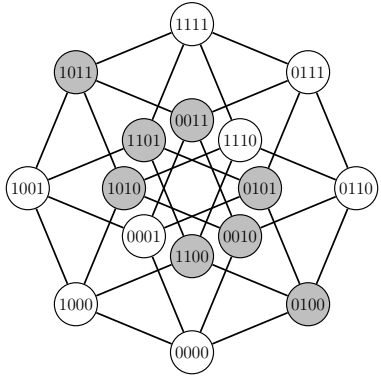
\begin{theorem}\label{thm:cart_prod}
Let $G$ be a graph with minimal fort $F$ such that no vertices in $F$ are adjacent.
Also, let $G'$ be a graph with minimal fort $F'$.
Then, $\hat{F}=F\times F'$ is a minimal fort of $\hat{G}=G\Box G'$. 
\end{theorem}
\begin{proof}
    By~\cite[Proposition 4.1]{Cameron2023}, $\hat{F}$ is a fort of $\hat{G}$; hence, we only need to show minimality. 
    To this end, we show that every proper non-empty subset of $\hat{F}$ is not a fort of $\hat{G}$.
    Let $\emptyset\subsetneq\hat{S}\subsetneq\hat{F}$ and for each $i\in F'$ define $S_{i}=\left\{v\in F\colon (v,i)\in\hat{S}\right\}$.
    We break our analysis into two cases: First, where $\emptyset\subsetneq S_{i}\subsetneq F$ for some $i\in F'$; second, where for all $i\in F'$, $S_{i}=\emptyset$ or $S_{i}=F$.

    In this first case, there is an $i\in F'$ such that $\emptyset\subsetneq S_{i}\subsetneq F$.
    Since $F$ is a minimal fort of $G$, there is a $u\in V(G)\setminus{S_{i}}$ and $v\in S_{i}$ such that $N_{G}(u)\cap S_{i} = \{v\}$.
    Since $F$ is assumed to have no adjacent neighbors, it follows that $u\notin F$.
    Therefore, $(u,v')\notin\hat{F}$ for all $v'\in F'$, which implies that $N_{\hat{G}}\left((u,i)\right)\cap\hat{S}=\left\{(v,i)\right\}$.
    Hence, $\hat{S}$ is not a fort of $\hat{G}$. 

    In the second case, $S_{i}$ is either empty or all of $F$ for each $i\in F'$.
    Let $\mathcal{I}\subset F'$ denote the set of vertices such that $S_{i}=\emptyset$ for all $i\in\mathcal{I}$.
    Since $\emptyset\subsetneq\hat{S}\subsetneq\hat{F}$, it follows that $\emptyset\subsetneq\mathcal{I}\subsetneq F'$. 
    Moreover, since $F'$ is minimal, there is a $u'\in V(G')\setminus\left(F'\setminus\mathcal{I}\right)$ and a $v'\in F'\setminus\mathcal{I}$ such that $N_{G'}(u')\cap\left(F'\setminus\mathcal{I}\right)=\{v'\}$. 
    Note that $N_{\hat{G}}\left((u,u')\right)\cap\hat{S}=\left\{(u,v')\right\}$ for each $u\in V(G)$.
    Therefore, for any $u\in F$, $(u,u')$ is a vertex outside of $\hat{S}$ with only one neighbor in $\hat{S}$, namely $(u,v')$.
    Hence, $\hat{S}$ is not a fort of $\hat{G}$. 
\end{proof}

Next, we consider the construction of minimal forts of $G\Box G'$, where $G$ and $G'$ are bipartite graphs. 
In this case, we are able to swap the non-adjacency condition of Theorem~\ref{thm:cart_prod} with a condition that requires every vertex in $F$ and $F'$ to have a neighbor in $F$ and $F'$, respectively. 
As an example, let $F$ denote the minimal fort of $Q_{3}$ shown in Figure~\ref{fig:q3_minimalfort}.
If we denote the partitions of $Q_{3}$ by $V_{e}$ and $V_{o}$, see Observation~\ref{obs:bipartite}, then $F\cap V_{e}=\{000,011\}$ and $F\cap V_{o}=\{100,111\}$.
Now, let $F'=V(Q_{1})$ denote the only fort of $Q_{1}$ and let $V_{e}'$ and $V_{o}'$ denote the partitions of $Q_{1}$; then, 
\[
\hat{F} = \left(F\cap V_{e}\right)\times\left(F'\cap V_{e}'\right)\cup\left(F\cap V_{o}\right)\times\left(F'\cap V_{o}'\right)
\]
is the minimum fort of $Q_{4}$ shown in Figure~\ref{fig:q4_minfort}.
\begin{theorem}\label{thm:bip_cart_prod}
Let $G=(V,E)$ be a bipartite graph with partition $V=V_{1}\sqcup V_{2}$ and let $G'=(V',E')$ be a bipartite graph with partition $V'=V'_{1}\sqcup V'_{2}$.
Also, let $F$ and $F'$ be forts of $G$ and $G'$, respectively.
If every vertex of $F$ and $F'$ has a neighbor in $F$ and $F'$, respectively, then 
\[
\hat{F}=\left(F\cap V_{1}\right)\times\left(F'\cap V'_{1}\right)\cup\left(F\cap V_{2}\right)\times\left(F'\cap V'_{2}\right)
\]
is a fort of $\hat{G}=G\Box G'$. 
\end{theorem}
\begin{proof}
Let $(u,u')\in V(\hat{G})\setminus{\hat{F}}$ such that there exists a vertex
\[
(v,v')\in N_{\hat{G}}\left((u,u')\right)\cap\hat{F}.
\]
Since $(v,v')\in\hat{F}$, it follows that $v\in F\cap V_{1}$ and $v'\in F'\cap V'_{1}$ or $v\in F\cap V_{2}$ and $v'\in F'\cap V'_{2}$.
Without loss of generality, we assume that $v\in F\cap V_{1}$ and $v'\in F'\cap V'_{1}$.
Furthermore, since $(u,u')$ and $(v,v')$ are adjacent in $\hat{G}$, it follows that either $u=v$ and $u'\in N_{G'}(v')$ or $u'=v'$ and $u\in N_{G}(v)$.
Without loss of generality, we assume that $u=v$ and $u'\in N_{G'}(v')$.
Now, we only need to consider two cases: Either $u'\in F'$ or $u'\notin F'$.

If $u'\in F'$, then $u'\in F'\cap V'_{2}$ since $G'$ is bipartite and $u'$ is adjacent to $v'$, which is in $V'_{1}$.
Furthermore, since every vertex in $F$ is adjacent to another vertex in $F$, there exists a $w\in F\cap V_{2}$ such that $w\in N_{G}(v)=N_{G}(u)$.
Note that $(w,u')\in\hat{F}$ and $(w,u')\in N_{\hat{G}}\left((u,u')\right)$.

If $u'\notin F'$, then there exists a $w'\in F'\setminus\{v'\}$ such that $u'\in N_{G'}(v')$ since $F'$ is a fort of $G'$. 
Note that $(v,w')\in\hat{F}$ and $(v,w')\in N_{\hat{G}}\left((u,u')\right)$. 
\end{proof}

The fort $\hat{F}$  from Theorem~\ref{thm:bip_cart_prod} may not be minimal, even if we restrict $F$ and $F'$ to be minimal. 
For example, let $G$ and $F$ denote the graph and minimal fort, respectively, shown in Figure~\ref{fig:bipartite}.
If we let $V_{1}$ and $V_{2}$ denote the vertex partitions of $G$, then 
\[
\hat{F} = \left(F\cap V_{1}\right)\times\left(F\cap V_{1}\right)\cup\left(F\cap V_{2}\right)\times\left(F\cap V_{2}\right)
\]
is a fort of $G\Box G$ but it is not minimal since $(1,1)\in\hat{F}$ and $\hat{F}\setminus\{(1,1)\}$ is a fort. 
Indeed, $N((1,1))=\{(1,2),(2,1)\}$ and both vertices $(1,2)$ and $(2,1)$ have three neighbors in $\hat{F}$.
Therefore, removing $(1,1)$ from $\hat{F}$ will not result in any possible forcings.
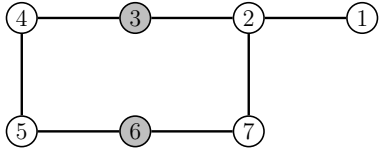
\begin{figure}[ht]
\centering
\resizebox{0.3\textwidth}{!}{%
\begin{tikzpicture}
[nodeFilledDecorate/.style={shape=circle,inner sep=2pt,draw=black,fill=lightgray,thick},%
    nodeEmptyDecorate/.style={shape=circle,inner sep=2pt,draw=black,thick},%
    lineDecorate/.style={black,=>latex',very thick},%
    scale=2.0]
    \foreach \nodename/\x/\y in {3/-1/0, 6/-1/-1}
    {
        \node (\nodename) at (\x,\y) [nodeFilledDecorate] {\nodename};
    }
    \foreach \nodename/\x/\y in {1/1/0, 2/0/0, 4/-2/0, 5/-2/-1, 7/0/-1}
    {
        \node (\nodename) at (\x,\y) [nodeEmptyDecorate] {\nodename};
    }
    \path\foreach \startnode/\endnode in {1/2, 2/3, 3/4, 4/5, 5/6, 6/7, 7/2}
    {
        (\startnode) edge[lineDecorate] node {} (\endnode)
    };
\end{tikzpicture}%
}
\caption{A minimal fort (white) where each vertex in the fort has at least one neighbor in the fort.}
\label{fig:bipartite}
\end{figure}

In the following theorem, we show that the fort $\hat{F}$ will be minimal if we require both $F$ and $F'$ to be minimal, every vertex in $F$ must have a neighbor in $F$, and every vertex in $F'$ must have exactly one neighbor in $F'$.
\begin{theorem}\label{thm:bip_cart_prod_minimal}
Let $G=(V,E)$ be a bipartite graph with partition $V=V_{1}\sqcup V_{2}$ and let $G'=(V',E')$ be a bipartite graph with partition $V'=V'_{1}\sqcup V'_{2}$.
Also, let $F$ and $F'$ be minimal forts of $G$ and $G'$, respectively.
If every vertex of $F$  has at least one neighbor in $F$ and every vertex of $F'$ has exactly one neighbor in $F'$, then 
\[
\hat{F}=\left(F\cap V_{1}\right)\times\left(F'\cap V'_{1}\right)\cup\left(F\cap V_{2}\right)\times\left(F'\cap V'_{2}\right)
\]
is a minimal fort of $\hat{G}=G\Box G'$. 
\end{theorem}
\begin{proof}
By Theorem~\ref{thm:bip_cart_prod}, $\hat{F}$ is a fort of $\hat{G}$; hence, we only need to show minimality.
To this end, we show that every proper non-empty subset of $\hat{F}$ is not a fort of $\hat{G}$.
Let $\emptyset\subsetneq\hat{S}\subsetneq\hat{F}$.
For each $i\in F'$ there is a unique $j\in F'$ such that $i$ and $j$ are adjacent in $G'$. 
Now, for each neighbor pair $i,j\in F'$, define $S_{ij}=\left\{v\in F\colon (v,i)\in\hat{S}~\textrm{or}~(v,j)\in\hat{S}\right\}$.
We break our analysis into two cases: First, where $\emptyset\subsetneq S_{ij}\subsetneq F$ for some neighbor pair $i,j\in F'$; second, where for all neighbor pairs $i,j\in F'$, $S_{ij}=\emptyset$ or $S_{ij}=F$. 

In the first case, there is a neighbor pair $i,j\in F'$ such that $\emptyset\subsetneq S_{ij}\subsetneq F$. 
Since $F$ is a minimal fort of $G$, there is a $u\in V(G)\setminus{S_{ij}}$ and $v\in S_{ij}$ such that $N_{G}(u)\cap S_{ij}=\{v\}$. 
Without loss of generality, assume that $(v,i)\in\hat{S}$.
Then, since $j$ is the only vertex in $F'$ that is adjacent to $i$, it follows that
\[
N_{\hat{G}}\left((u,i)\right)\cap\hat{S} = \left\{(v,i)\right\},
\]
so $\hat{S}$ is not a fort of $\hat{G}$.

In the second case, for each neighbor pair $i,j\in F'$, $S_{ij}=\emptyset$ or $S_{ij}=F$.
Let $\mathcal{I}\subset F'$ denote the set of vertices $i\in F'$ such that $S_{ij}=\emptyset$, where $j$ is in the unique neighbor of $i$ in $F'$. 
Since $\emptyset\subsetneq\hat{S}\subsetneq\hat{F}$, it follows that $\emptyset\subsetneq\mathcal{I}\subsetneq F'$.
Moreover, since $F'$ is minimal, there is a $u'\in V(G')\setminus\left(F'\setminus\mathcal{I}\right)$ and a $v'\in F'\setminus\mathcal{I}$ such that $N_{G'}(u')\cap\left(F'\setminus\mathcal{I}\right)=\{v'\}$.
Note that 
\[
N_{\hat{G}}\left((u,u')\right)\cap\hat{S} = \left\{(u,v')\right\},
\]
for each $u\in V(G)$.
Therefore, for any $u\in F$, $(u,u')$ is a vertex outside of $\hat{S}$ with only one neighbor in $\hat{S}$, namely $(u,v')$.
Hence, $\hat{S}$ is not a fort of $\hat{G}$. 
\end{proof}

We conclude this section by showing that Proposition~\ref{prop:nbhd_fort}, Theorem~\ref{thm:cart_prod}, and Theorem~\ref{thm:bip_cart_prod_minimal} can be used to construct some, but not all, of the minimal forts of the hypercube graph. 
To begin, there are only $2$ minimal forts on $Q_{2}$, both of which are neighborhoods. 

On $Q_{3}$, there are $2^{3}$ minimal (and also minimum) forts that are neighborhoods. 
Furthermore, we can use the neighborhood forts on $Q_{2}$ and Theorem~\ref{thm:cart_prod} to construct other minimal forts of $Q_{3}$. 
For example, $F=N(00)=\{01,10\}$ is a minimal fort of $Q_{2}$ where no vertex in $F$ has a neighbor in $F$; hence, by Theorem~\ref{thm:cart_prod}, $\hat{F}=N(00)\times\{0,1\}=\{010,100,011,101\}$ is a minimal (but not minimum) fort of $Q_{3}$. 
Moreover, every subset of $V(Q_{3})$ that is automorphic to $\hat{F}$ is also a minimal fort, there are $6$ such distinct automorphic subsets, counting $\hat{F}$. 
Therefore, $Q_{3}$ has at least $2^{3}+6=14$ minimal forts.
One can verify computationally that this enumerates all minimal forts of $Q_{3}$. 

On $Q_{4}$, there are $2^{4}$ minimal forts that are neighborhoods. 
Again, we can use the neighborhood forts on $Q_{3}$ and Theorem~\ref{thm:cart_prod} to construct other minimal forts on $Q_{4}$.
For example, $F=N(000)=\{010,100,001\}$ is a minimal fort of $Q_{3}$ where no vertex in $F$ has a neighbor in $F$; hence, by Theorem~\ref{thm:cart_prod}, $\hat{F}=N(000)\times\{0,1\}=\{0100,1000,0010,0101,1001,0011\}$ is a minimal fort of $Q_{4}$.
Moreover, every subset of $V(Q_{4})$ that is automorphic to $\hat{F}$ is a minimal fort, and there are $32$ such distinct automorphic subsets, counting $\hat{F}$.

In addition, we can apply Theorem~\ref{thm:bip_cart_prod_minimal} to the minimal forts on $Q_{3}$ that were constructed from minimal forts on $Q_{2}$ to produce minimal forts of $Q_{4}$.
For example, $F=\{010,100,011,101\}$ is a minimal fort of $Q_{3}$ where every vertex in $F$ has at least one neighbor in $F$.
If we let $V_{e}$ and $V_{o}$ denote the partitions of $Q_{3}$, then Theorem~\ref{thm:bip_cart_prod_minimal} implies that 
\[
\hat{F} = \left(F\cap V_{e}\right)\times\{0\}\cup\left(F\cap V_{o}\right)\times\{1\}
\]
is a minimal fort of $Q_{4}$.
Again, every subset of $V(Q_{4})$ that is automorphic to $\hat{F}$ is a minimal fort, there are $12$ such distinct automorphic subsets, counting $\hat{F}$. 
Therefore, $Q_{4}$ has at least $2^{4}+32+12=60$ minimal forts. 
In this case, we have not identified a complete enumeration of the minimal forts on $Q_{4}$; for example, we are missing minimal forts of the form in Figure~\ref{fig:q4_other_minimal_fort}. 

Similar analysis can be used to partially characterize the minimal forts of larger hypercubes, but a full characterization of all the minimal forts of $Q_d$ is still an open problem.
\begin{figure}[ht]
\centering
\resizebox{0.3\textwidth}{!}{%
\begin{tikzpicture}
[nodeFilledDecorate/.style={shape=circle,inner sep=2pt,draw=black,fill=lightgray,thick},%
    nodeEmptyDecorate/.style={shape=circle,inner sep=2pt,draw=black,thick},%
    lineDecorate/.style={black,=>latex',very thick},%
    scale=4.0]
    \foreach \nodename/\x/\y in {0111/0.7071/0.7071, 1001/-1/0, 0100/0.7071/-0.7071, 0011/0/0.4142, 0010/0.2928/-0.2928, 1101/-0.2928/0.2928, 1010/-0.4142/0, 1100/0/-0.4142}
    {
        \node (\nodename) at (\x,\y) [nodeFilledDecorate] {\nodename};
    }
    \foreach \nodename/\x/\y in {1011/-0.7071/0.7071, 0101/0.4142/0, 0001/-0.2928/-0.2928, 1110/0.2928/0.2928, 1000/-0.7071/-0.7071, 0000/0/-1, 1111/0/1, 0110/1/0}
    {
        \node (\nodename) at (\x,\y) [nodeEmptyDecorate] {\nodename};
    }
    \path\foreach \startnode/\endnode in {0110/0111, 0110/0100, 0110/1110, 0110/0010, 0111/1111, 0111/0101, 0111/0011, 1111/1011, 1111/1110, 1111/1101, 1011/1001, 1011/0011, 1011/1010, 1001/1000, 1001/1101, 1001/0001, 1000/0000, 1000/1010, 1000/1100, 0000/0100, 0000/0001, 0000/0010, 0100/0101, 0100/1100, 0101/1101, 0101/0001, 1110/1010, 1110/1100, 0011/0001, 0011/0010, 1101/1100, 1010/0010}
    {
        (\startnode) edge[lineDecorate] node {} (\endnode)
    };
\end{tikzpicture}%
}
\caption{A minimal fort of $Q_{4}$ not attained via the constructions in Proposition~\ref{prop:nbhd_fort}, Theorem~\ref{thm:cart_prod}, nor Theorem~\ref{thm:bip_cart_prod_minimal}.}
\label{fig:q4_other_minimal_fort}
\end{figure}
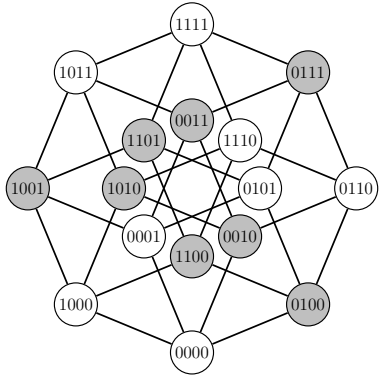
\section{Conclusion}\label{sec:conclusion}
This article provides a comprehensive characterization of the minimum forts of the hypercube graph $Q_{d}$. Specifically, Theorem~\ref{thm:min_fort} establishes that for dimensions $d \geq 2$ and $d \neq 4$, every minimum fort of $Q_{d}$ is the neighborhood of some vertex $v \in Q_{d}$. Furthermore, Proposition~\ref{prop:min_fort4} demonstrates that the minimum forts of $Q_{4}$ are either neighborhoods or are automorphic to the set depicted in Figure~\ref{fig:q4_minfort}.
As a consequence of this characterization, we derive the fractional zero forcing number of the hypercube in Corollary~\ref{cor:fzf_hypercube}, along with upper and lower bounds on the fort number of the hypercube in Corollary~\ref{cor:ft_bounds}. Additionally, when the dimension of the hypercube is a power of two, we establish that the fort number and fractional zero forcing number are equal to the domination number, total domination number, and the open packing number.

While the minimum forts of the hypercube are highly symmetric, the minimum zero forcing sets are highly non-symmetric. In particular, Corollary~\ref{cor:pt1-4} indicates that for $d \geq 4$, $Q_{d}$ has minimum zero forcing sets with propagation times of $1$, $2$, $3$, and $4$. Moreover, we conjecture that $Q_d$ has minimum zero forcing sets with propagation times of $1, 2, \ldots, 2^{d-2}$, but this remains an open problem, as is the complete enumeration of all the minimum zero forcing sets of $Q_d$.


Lastly, we present general constructions for minimal forts in the Cartesian product of graphs. In particular, Theorem~\ref{thm:bip_cart_prod} establishes that $F \times F'$ is a minimal fort of $G \Box G'$ provided that $F$ is a minimal fort of $G$ containing no adjacent vertices, and $F'$ is a minimal fort of $G'$. Moreover, Theorem~\ref{thm:bip_cart_prod_minimal} shows that when both $G$ and $G'$ are bipartite, minimal forts for $G \Box G'$ can be constructed from minimal forts $F$ and $F'$ of $G$ and $G'$, respectively, given that every vertex in $F$ has at least one neighbor in $F$ and every vertex in $F'$ has exactly one neighbor in $F'$.

While these constructions can be applied generally, they are particularly useful for constructing minimal forts of the hypercube graph, as detailed in Section~\ref{sec:minimal_forts}. However, these methods do not enumerate all minimal forts of the hypercube graph. For example, for $Q_{4}$, we can construct $60$ minimal forts using the above results, but computational methods reveal that there are $348$ minimal forts. Thus, it remains an open problem to either characterize all minimal forts or to provide bounds on the number of minimal forts of the hypercube graph.


\begin{thebibliography}{30}
\bibitem{Aazami2008}
{\sc A. Aazami},
\newblock {\em Hardness results and approximation algorithms for some problems on graphs},
\newblock University of Waterloo, 2008. 

\bibitem{Afzali2024}
{\sc F. Afzali, A.H. Ghodrati, and H.R. Maimani},
\newblock Failed zero forcing numbers of Kneser graphs, Johnson graphs, and hypercubes,
\newblock {\em  J. Appl. Math. Comput.}, 70:2665--2675, 2024.

\bibitem{AIM2008}
{\sc AIM Minimum Rank – Special Graphs Work Group},
\newblock Zero forcing sets and the minimum rank of graphs,
\newblock {\em Linear Algebra Appl.}, 428:1628--1648, 2008.

\bibitem{Azarija2017}
{\sc J. Azarija, M.A. Henning, and S. Klav\v zar},
\newblock (Total) Domination in Prisms,
\newblock {\em  Electron. J. Combin.}, 24(1):19, 2017.

\bibitem{Bertolo2004}
{\sc R. Bertolo,  P.R.J. \" Osterg\aa rd, and W.D. Weakley}
\newblock An Updated Table of Binary/ternary Mixed Covering Codes,
\newblock {\em J. Combin. Des.}, 12:157-176, 2004.

\bibitem{Bresar2024}
{\sc B. Bre\v sar, S. Klav\v zar, D.F. Rall}
\newblock Packings in bipartite prisms and hypercubes,
\newblock {\em Discrete Math.}, 347:113875, 2024.

\bibitem{Brimkov2019}
{\sc B. Brimkov, C.C. Fast, and I.V. Hicks},
\newblock Computational approaches for zero forcing
and related problems,
\newblock {\em European J. Oper. Res.}, 273(3):889--903, 2019.

\bibitem{Brimkov2021}
{\sc B. Brimkov, D. Mikesell, and I.V. Hicks},
\newblock Improved computational approaches and heuristics for zero forcing,
\newblock {\em INFORMS J. Comput.}, 33(4):1384-1399, 2021.

\bibitem{Brimkov2019th}
{\sc B. Brimkov, J. Carlson, I.V. Hicks, R. Patel, and L. Smith},
\newblock Power domination throttling,
\newblock {\em Theoret. Comput. Sci.}, 795:142--153, 2019. 

\bibitem{Brueni2005}
{\sc D.J. Brueni and L.S. Heath},
\newblock The {PMU} placement problem,
\newblock {\em SIAM J. Discrete Math.}, 19(3):744--761, 2005.

\bibitem{Burgarth2007}
{\sc D. Burgarth and V. Giovannetti},
\newblock Full control by locally induced relaxation,
\newblock {\em Phys. Rev. Lett.}, 99:100501, 2007.

\bibitem{Butler2013}
{\sc S.~Butler and M.~Young},
\newblock Throttling zero forcing propagation time speed on graphs,
\newblock {\em Australas J. Combin.}, 57:65--71, 2013. 

\bibitem{Cameron2023}
{\sc T.R. Cameron, L. Hogben, F.H.J. Kenter, S.A. Mojallal, and H. Schuerger},
\newblock Forts, (fractional) zero forcing, and Cartesian products of graphs,
\newblock {\em arXiv:2310.17904 [math.CO]}, 2023.

\bibitem{Carlson2021}
{\sc J. Carlson and J. Kritschgau},
\newblock Various characterizations of throttling numbers,
\newblock {\em Discrete Appl. Math.}, 294:85--97, 2021.

\bibitem{Chilakamarri2012}
{\sc K. Chilakamarri, N. Dean, C.X. Kang, and E. Yi},
\newblock Iteration index of a zero forcing set in a graph,
\newblock {\em Bull. Inst. Combin. Appl.}, 64:57--72, 2012. 

\bibitem{Erdos1973}
{\sc P. Erd\H os and R.K. Guy},
\newblock Crossing Number Problems,
\newblock {\em Amer. Math. Monthly}, 80:52--58, 1973.

\bibitem{Fast2018}
{\sc C.C. Fast and I.V. Hicks},
\newblock Effects of vertex degrees on the zero-forcing number and
propogation time of a graph,
\newblock {\em Discrete Appl. Math.}, 250:215--226, 2018.

\bibitem{Fetcie2015}
{\sc K. Fetci, B. Jacob, and D. Saavedra},
\newblock The failed zero forcing number of a graph,
\newblock {\em Involve}, 8(1):99--117, 2015.

\bibitem{Foldes1977}
{\sc S. Foldes},
\newblock A characterization of hypercubes,
\newblock {\em Discrete Math.}, 17(2):155--159, 1977.

\bibitem{Gomez2021}
{\sc L. Gomez, K. Rubi, J. Terrazas, and D.A. Narayan},
\newblock All graphs with a failed zero forcing number of two,
\newblock {\em Symmetry}, 13(11):2221, 2021. 

\bibitem{Gomez2025}
{\sc L. Gomez, K. Rubi, J. Terrazas, and D.A. Narayn},
\newblock Failed zero forcing number of trees and circulant graphs,
\newblock {\em Theory and Applications of Graphs}, 11(1):5, 2025.

\bibitem{Hogben2012}
{\sc L. Hogben, M. Huynh, N. Kingsley, S. Meyer, S. Walker, and M. Young},
\newblock Propagation time for zero forcing on a graph,
\newblock {\em Discrete Appl. Math.}, 160(13):1994--2005, 2012.

\bibitem{Hogben2022}
{\sc L. Hogben, J.C.-H. Lin, and B. Shader},
\newblock {\em Inverse Problems and Zero Forcing for graphs},
\newblock Mathematical Surveys and Monographs 270, American Mathematical Society, Providence, RI, 2022.

\bibitem{Jenssen2022}
{\sc M. Jenssen, W. Perkins, and A. Potukuchi},
\newblock Independent sets of a given size and structure in the hypercube,
\newblock {\em Combin. Probab. Comput.} 31(4):702-720, 2022.

\bibitem{Kaudan2023}
{\sc C. Kaudan, R. Taylor, D.A. Narayan},
\newblock An inverse approach for finding graphs with a failed zero forcing number of k,
\newblock {\em Mathematics} 11(19):4068, 2023.

\bibitem{Mirafzal2025}
{\sc S.M. Mirafzal},
\newblock On the distance-transitivity of the folded hypercube,
\newblock {\em Commun. Comb. Optim.} 10(1):207--217, 2025. 

\bibitem{Severini2008}
{\sc S. Severini},
\newblock Nondiscriminatory propagation on trees,
\newblock {\em J. Phys. A}, 41(48):482002, 2008.

\bibitem{Shitov2017}
{\sc Y. Shitov},
\newblock On the complexity of failed zero forcing,
\newblock {\em Theoret. Comput. Sci.}, 660:102--104, 2017. 

\bibitem{Swanson2023}
{\sc N. Swanson and E. Ufferman},
\newblock A lower bound on the failed zero-forcing number of a graph,
\newblock {\em Involve}, 16(3):493--504, 2023. 

\bibitem{West2001}
{\sc D. West},
\newblock {\em Introduction to Graph Theory, 2nd ed.},
\newblock Prentice Hall, Upper Saddle River, NJ, 2001.
\end{thebibliography}
\end{document}